\documentclass[12pt]{article}
\usepackage{url}
\usepackage{fancyhdr}
\usepackage{amsmath, amssymb, latexsym, amscd, amsthm,amsfonts,amstext}
\usepackage{subfigure}
\usepackage{enumerate}
\usepackage{xcolor}
\usepackage{textcomp}
\usepackage{mathtools}

\def\diag{\operatorname{diag}}
\usepackage{graphicx}
\newtheorem{theorem}{Theorem}[section]
\newtheorem{lemma}[theorem]{Lemma}

\numberwithin{equation}{section}
\newtheorem{definition}[theorem]{Definition}
\newtheorem{remark}[theorem]{Remark}

\numberwithin{equation}{section}

\usepackage[mathscr]{eucal}
\usepackage{graphicx}
\usepackage{subfigure}

\usepackage{fancyhdr}

\usepackage{enumerate}
\usepackage{xcolor}

\allowdisplaybreaks

\title{Extraction of the index of refraction by embedding multiple small
inclusions} 
\author{Ahmed Alsaedi
\thanks{Nonlinear Analysis and Applied Mathematics Research Group (NAAM), Department of  Mathematics,
Faculty of Sciences,
King Abdulaziz University,
P.O. Box 80203,
Jeddah 21589,
Saudi Arabia.}
\and
Faris Alzahrani\footnotemark[1]
\and  
 Durga Prasad Challa 
\thanks{Department of mathematics, 
Tallinn University of Technology, Tallinn 19068, Estonia.
(Email: durga.challa@ttu.ee).
}
\and Mokhtar Kirane\footnotemark[1]
\thanks{Laboratoire de Math\'ematiques,
P\^ole Sciences et Technologies, 
Universit\'e de La Rochelle, 
Avenue Michel Cr\'epeau
17042, La Rochelle Cedex, France, (Email: mokhtar.kirane@univ-lr.fr).}
\and  
Mourad Sini
 \thanks{RICAM, Austrian Academy of Sciences,
Altenbergerstrasse 69, A-4040, Linz, Austria.
(Email: mourad.sini@oeaw.ac.at).
}}
 
\begin{document}
 \maketitle
\begin{abstract}
We deal with the problem of reconstructing material coefficients from the far-fields 
they generate. 
By embedding small (single) inclusions to these media, located at points $z$ in the 
support of these materials, and measuring the far-fields generated by these deformations 
we can extract the values of the total field (or the energies) generated by these media at the points $z$. 
The second step is to extract the values of the material coefficients from these internal
values of the total field. The main difficulty in using internal fields is the
treatment of their possible zeros.   

In this work, we propose to deform the medium using multiple (precisely double) and close inclusions 
instead of only single ones. By doing so, we derive 
from the asymptotic expansions of the far-fields the internal values of the Green function, in addition 
to the internal values of the total fields. This is possible because of the deformation of the 
medium with multiple inclusions which generates scattered fields due to the 
multiple scattering between these inclusions.  Then, the values of the index of refraction
can be extracted from the singularities of the Green function. Hence, we overcome the 
difficulties arising from the zeros of the internal fields.

We test these arguments for the acoustic scattering by a refractive index in presence of 
inclusions modeled by the impedance type small obstacles. 

 \end{abstract}
 \textbf{Keywords}: Inverse acoustic scattering, small inclusions, multiple scattering, refraction index.



\section{Introduction and statement of the results}\label{Introduction-smallac-sdlp}

\subsection{Motivation of the problem}

Let $n$ be a bounded and measurable function in $\mathbb{R}^3$ such that the support of 
$(n-1)$ is a bounded domain $\Omega$.
We are concerned with the acoustic scattering problem 
\begin{equation}
(\Delta + \kappa^{2}n^2(x))V^{t}=0 \mbox{ in }\mathbb{R}^{3},\label{acimpoenetrable-2}
\end{equation}
where $V^t:=V^s+V^i$ and $V^i$ is an incident field satisfying $(\Delta +\kappa^2)V^i=0$ in $\mathbb{R}^3$. 
For simplicity, we take incident plane waves $V^i(x, \theta):=e^{i\kappa x \cdot \theta}$ where 
$\theta \in \mathbb{S}^2$ and $\mathbb{S}^2$ is the unit sphere. The scattered field $V^s$ satisfies the 
Sommerfeld radiation condition:
\begin{equation}\label{radiationc-2}
\frac{\partial V^{s}}{\partial |x|}-i\kappa V^{s}=o\left(\frac{1}{|x|}\right), |x|\rightarrow\infty. 
\end{equation}

The scattering problem (\ref{acimpoenetrable-2}-\ref{radiationc-2}) is well posed, see \cite{C-K:1998}. 
Applying Green's formula to $V^s$, we can show that the scattered field $V^s(x, \theta)$ has the following 
asymptotic expansion:
\begin{equation}\label{far-field-n}
 V^s(x, \theta)=\frac{e^{i \kappa |x|}}{4\pi|x|}V^{\infty}(\hat{x}, \theta) + O(|x|^{-2}), \quad |x|
\rightarrow \infty,
\end{equation}
where the function $V^{\infty}(\hat{x}, \theta)$ for $(\hat{x}, \theta)\in \mathbb{S}^{2} \times\mathbb{S}^{2}$  
is the corresponding far-field pattern.
\par Our interest in this work is related to the classical inverse scattering problem 
which consists of
reconstructing $n(x),\; x \in \Omega,$ from the far-field data 
$V^{\infty}(\hat{x}, \theta)$ for some 
$(\hat{x}, \theta)\in \mathbb{S}^{2} \times\mathbb{S}^{2}$. This type of problem is well
studied and there are 
several algorithms to solve it in the case when $\hat{x}$ and $\theta$ are taken in the
whole $ \mathbb{S}^{2}$, 
see \cite{Nachman:1988, Novikov:1988, Ramm:1988}. It is also known that this problem is 
very unstable. 
Precisely, the modulus of continuity is in general of logarithmic type, 
see \cite{MRAMAG-book1}.

Recently, based on a new type of experiments, a different approach was proposed, see for
instance \cite{A-B-C-T-F:2008, A-C-D-R-T} and the references therein. It is divided into two steps. In the first one, we deform the 
acoustic medium by small inclusions located in the region containing the support of
$(n-1)$, and measure the far-fields generated by these deformations. 
From these far-fields, we can extract the total fields due to the medium $n$ in the 
interior of the support of $(n-1)$.  In the second step, we reconstruct $n(x)$ from 
these interior values of the total fields. Recently, there was an increase of interest 
in reconstructing media from internal measurements. This is also related to the hybrid 
methods introduced in the medical imaging community, see \cite{A:2008, B-B-M-T, B:2014, F-T:2010,
F-T:2011, L-F:2013} for different setups and models. In contrast to 
the instability of the classical inverse scattering problem, the reconstruction from 
internal measurements is stable, see \cite{Alessandrini:2014, A-G-W:2012, H-M-N:2014, Triki:2010}. 
However, the disadvantage with internal fields, i.e. internal values of the total fields,
is the existence of their zeros which need to be properly dealt with to stabilize any algorithm. 
To overcome this short coming, one can think of using measurements related to multiple 
frequencies $\kappa_1,\kappa_2, etc.$, see \cite{Alberti:2013}. 

Our objective in this paper is to propose an alternative to overcome the last 
disadvantage. For this, we propose to deform the medium using multiple (precisely double) and close 
inclusions instead of just single inclusions. Then,
measuring the generated far-fields by these multiple inclusions, we can extract, not only 
the internal total fields, but also the internal values of the Green's function related 
to the non deformed medium $n$. This is possible because by deforming the medium with
multiple and close inclusions we generate scattered fields due to the multiple scattering 
between these inclusions. Precisely, the far-fields we measure encode at least the second 
order term in the Foldy-Lax approximation and not only the first order (or the Born 
approximation) as it is done when deforming with single or multiple but well separated inclusions. 
Finally, we extract the values $n(x)$, $x$ in $\Omega$, from the singularities of this 
Green's function.  Hence, we avoid the problems coming from the
zeros of the internal total fields.   

The accuracy of the reconstruction is related to the minimum distance $d$ between the pairs of 
inclusions, in addition to the radius of the perturbations $a$. In practice, the perturbations are 
created using focusing waves, see \cite{A:2008, F-T:2010, L-F:2013} for instance, hence we are limited by the resolution of the used waves (acoustic waves for instance). In other words,
 the small perturbations cannot be very close. In our analysis, the ratio $\frac{a}{d}$ is 
of the form $a^{1-t}$, with a parameter $t\in [0, 1]$,  and it can be chosen small by taking 
$t$ near $0$. This allows to have the minimum distance between the inclusions quite large 
compared to their sizes. Since the accuracy of the reconstruction, see Theorem \ref{Main-results}, is of the order $a^t$, with
$t$ in $[0, t^*]$ with some $ t^* <1$, then one should find a reasonable balance between the limited resolution
of the focusing waves, to be used to perturb the medium, and the desired accuracy 
of the reconstruction. More details related to this issue are provided in Remark \ref{Remark}.

We complete this section by describing the  kind of deformations and the corresponding far-field measurements we use and then we state the derived formulas to extract the values of the index of refraction $n$. In Section $2$, we justify these formulas with the emphasize on the explicit dependence of the error terms on the parameters modeling the small perturbations (i.e. the maximum radii, the minimum distance between them and the scaled surface impedances). In Section $3$, we discuss the stability issue by providing the corresponding formulas when the measured data are contaminated with additive noises and when the small perturbations are shifted from their exact locations (as moving from step $2$ to step $3$ in collecting the data, see Section $1.3$). We provide the regimes under which the derived formulas are still valid. 

\subsection{Deformation by multiple inclusions}
We give the details of this approach by using inclusions of the form of obstacles 
of impedance type. 
Let $B$ be the ball with the center at the origin and radius $1$.  
We set $D_m:=a B+z_m$ to be the small bodies characterized by the parameter 
$a>0$ and the locations $z_m\in \mathbb{R}^3$, $m=1,\dots,M$.  
We denote by  $U^{s}$ the acoustic field scattered by the $M$ small bodies 
$D_m\subset \mathbb{R}^{3}$, due to 
the incident field $U^{i}$ (mainly the plane incident waves 
$U^{i}(x,\theta):=V^i(x, \theta):=e^{ikx\cdot\theta}$
with the incident direction $\theta \in \mathbb{S}^2$, where $\mathbb{S}^2$ 
being the unit sphere), 
with impedance boundary conditions. Hence the total field $U^{t}:=U^{i}+U^{s}$ 
satisfies the 
following exterior impedance problem for the acoustic waves

\begin{equation}
(\Delta + \kappa^{2}n^2(x))U^{t}=0 \mbox{ in }\mathbb{R}^{3}\backslash \left(\mathop{\cup}_{m=1}^M \bar{D}_m\right),\label{acimpoenetrable-1}
\end{equation}
\begin{equation}
\left.\frac{\partial U^{t}}{\partial \nu_m}+\lambda_m U^{t}\right|_{\partial D_m}=0,\, 1\leq m \leq M, \label{acgoverningsupport-1}  
\end{equation}
\begin{equation}
\frac{\partial U^{s}}{\partial |x|}-i\kappa U^{s}=o\left(\frac{1}{|x|}\right), |x|\rightarrow\infty. \label{radiationc-1}
\end{equation}
The scattering problem 
(\ref{acimpoenetrable-1})-(\ref{radiationc-1}) is well posed, see \cite{C-K:1983, C-K:1998}, and we can also allow $\Im \lambda_m$ to be negative, see \cite{C-S:2016}.

\bigskip

\begin{definition}\label{Def1}
We define 
\begin{enumerate}

 \item  $
d:=\min\limits_{\substack{m\neq j\\1\leq m,j\leq M }} d_{mj},$
where $\,d_{mj}:=dist(D_m, D_j)$. 

\item $\kappa_{\max}$ as the upper bound of the used wave numbers, i.e. $\kappa\in[0,\,\kappa_{\max}]$.
\bigskip

The distribution of the scatterers is modeled as follows:
\item the number $M~:=~M(a)~:=~O(a^{-s})\leq M_{max} a^{-s}$ with a given positive constant $M_{max}$.\\
\item the minimum distance $d~:=~d(a)~\approx ~a^t$, i.e. $d_{min} a^t \leq d(a) \leq d_{max}a^t $, with 
given positive constants $d_{min}$ and $d_{max}$. \\
 \item the surface impedance $\lambda_m~:=~\lambda_{m,0}a^{-\beta}$, where $\lambda_{m,0} \neq 0$ and 
 might be a complex number.
\end{enumerate}
\end{definition}
Here the real numbers $s$, $t$ and $\beta$ are assumed to be non negative.
We call $M_{max}, d_{min}, d_{max}$ and $\kappa_{max}$ 
the set of the a priori bounds. In (\cite{C-S:2016}, Corollary 1.3), we have shown that there exist positive constants 
$a_0$, $\lambda_-$, $\lambda_+$ \footnote{In the case where $s < 2-\beta$, we need no a priori conditions 
on $\lambda_-$, $\lambda_+$.} 
depending only on the set of the a priori bounds and on $n_{max}:=\Vert n\Vert_{L^\infty(\Omega)}$ such that
if 
\begin{equation} \label{conditions-1}
a \leq a_0,\; \vert \lambda_{m, 0}\vert \leq \lambda_+, \; \vert \Re(\lambda_{m, 0}) \vert \geq \lambda_-,\;~~ \beta \leq 1, \;~~ s \leq  2-\beta,\;~~\frac{s}{3}\leq t,
\end{equation} 
then the far-field pattern $U^\infty(\hat{x},\theta)$ has the following asymptotic expansion 

\begin{equation}\label{x oustdie1 D_m farmain-2}
U^\infty(\hat{x},\theta)=V^\infty(\hat{x},\theta)+ \sum_{m=1}^{M}V^t(z_m, -\hat{x})\mathbb{Q}_m+
O\left(M \;\mathcal{C}\;
 a\right),
 \end{equation}
uniformly in $\hat{x}$ and $\theta$ in $\mathbb{S}^2$. Here $\mathcal{C} :=\max_m C_m$.
The constant appearing in the estimate $O(.)$ depends only on 
the set of the a priori bounds, $\lambda_-$, $\lambda_+$ and on $n_{max}$. The coefficients $\mathbb{Q}_m$, $m=1,..., M,$ are the solutions of the following linear algebraic system
\begin{eqnarray}\label{fracqcfracmain-2}
 \mathbb{Q}_m +\sum_{\substack{j=1 \\ j\neq m}}^{M}C_m G(z_m,z_j)\mathbb{Q}_j&=&-C_mV^{t}(z_m, \theta),~~
\end{eqnarray}
for $ m=1,..., M,$ where
\begin{equation}\label{Cm-s-1}
C_m:=\frac{\lambda_m|\partial D_m|}{-1+\lambda_m I_m}
\end{equation}
and $G(\cdot, \cdot)$ is the Green function corresponding to the scattering problem 
(\ref{acimpoenetrable-2}-\ref{radiationc-2}).

The quantity $I_m:=\int_{\partial D_m}\frac{1}{4\pi\vert{s_m-t}\vert} d{s_m}, t \in \partial D_m,$ is a constant if $D_m$ is a ball and this constant is equal to the radius of $D_m$, i.e. $I_m=a$.
 The algebraic system \eqref{fracqcfracmain-2} is invertible under the condition:

\begin{eqnarray}\label{invertibilityconditionsmainthm-1}
 s \leq 2 -\beta.
\end{eqnarray} 

We use surface impedance functions of the form $\lambda_m=\lambda_{0, m} a^{-\beta}$, i.e. with $\lambda_{0, m} \neq 1$ in the case $\beta=1$, so that the constants 
\begin{equation}\label{C-m}
C_m=\frac{4\pi\; a^{2-\beta}\; \lambda_{0,m}}{-1+\lambda_{0,m} a^{1-\beta}}
\end{equation}
 are well defined.
In the subsequent sections, we choose $\beta =1$ and $\lambda_{0, m}:=1-a^h$, $h \geq 0$. In this case 
\begin{equation}\label{actual-capacitances}
C_m=-4\pi a^{1-h}(1-a^h).
\end{equation}

Hence (\ref{x oustdie1 D_m farmain-2}) becomes
\begin{equation}\label{x oustdie1 D_m farmain-3}
U^\infty(\hat{x},\theta)=V^\infty(\hat{x},\theta)+\sum_{m=1}^{M}V^t(z_m, -\hat{x})\mathbb{Q}_m+
O\left(a^{2-s-h}\right),
 \end{equation}
uniformly in $\hat{x}$ and $\theta$ in $\mathbb{S}^2$.

\subsection{The extraction formulas}\label{Sec:extractform} We proceed in three steps in collecting the measured data:
\begin{enumerate}
\item First step. We measure the far-fields before making any deformation. In this case the collected data are $ V^{\infty}(-\theta_i,\theta_j)$, $i,j=1,..., N_0 \geq 2$.
\bigskip

\item Second step. We deform the medium by 
\textcolor{black}{ a isolated inclusion $z^l_m$}, 
and then measure the corresponding data $ \textcolor{black}{U^{\infty}( -\theta_{i},\theta_{j}):=} U^{\infty}(z^l_m; -\theta_{i},\theta_{j})$, 
for the directions $\theta_{j}$, $j=1,...,N_0\geq2$, keeping in mind that we redo the experiment by moving the inclusion centered at $z^l_m$ for $m=1, 2,..., M$ and $l=1,2$, see Fig \ref{fig:1}.   

\begin{figure}[htp]
\centering
\includegraphics[width=4cm,height =4cm]{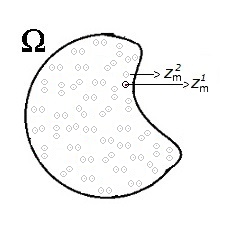}
\caption{The figure describes the experiments in step 2. 
We create only one inclusion and measure the corresponding far-fields. Then, we move the inclusion inside $\Omega$.}\label{fig:1}
\end{figure}
\item Third step. We deform the medium using two inclusions close to each other
i.e. the two points $z^1_m$ and $z^2_m$ 
in $\Omega$, are such that $\vert z^2_m-z^1_m \vert \sim a^t$, as $a \rightarrow 0$, 
with $t>0$, see Fig \ref{fig:2}. In this case, the collected data are 
$ \textcolor{black}{W^{\infty}( -\theta_{i},\theta_{j}):=} 
 W^{\infty}(z^l_m;-\theta_i,\theta_j)$, $i,j=1,..., N_0 \geq 2 $. As in step $2$, we create one couple of two close inclusions at once.
Then, we move the couple of inclusions, for $m=1, ..., M$, inside $\Omega$.
\begin{figure}[htp]
\centering
\includegraphics[width=4cm,height =4cm]{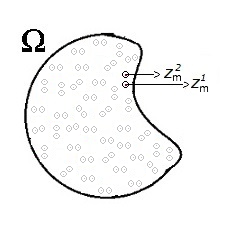}
\caption{\textcolor{black}{The figure describes the experiments in step 3. We create double inclusions distributed close to each other and we measure the corresponding far-fields. Then, we move the two inclusions inside $\Omega$}.}\label{fig:2}
\end{figure}
\end{enumerate}
\bigskip

The reconstruction formulas of the index of refraction from the described measured data are summarized in the following theorem. 
 
\begin{theorem}\label{Main-results} Assume $n$ to be a measurable and bounded function in $\mathbb{R}^3$ such that the support of 
$(n-1)$ is a bounded domain $\Omega$. 
Let the non negative parameters $\beta, h$ and $t$ describing the set of the 
small inclusions, embedded in $\Omega$, be such that
\begin{equation}\label{cond-beta-s-h-t}
 \beta =1,\; 0<h\; \mbox{ and }\; h+2t<1.
\end{equation}
Let $m=1, 2,..., M$ be fixed. We have the following formulas:
\begin{enumerate}
\item The vector of total fields $(V^t(z^l_m, \theta_j))^{N_0}_{j=1}$ (or $-(V^t(z^l_m, \theta_j))^{N_0}_{j=1}$) can be recovered from $U^{\infty}(z^l_m,\;-\theta_{i},\theta_{j})-V^{\infty}(-\theta_{i},\theta_{j})$ as follows:
\begin{equation}\label{total-fields}
C^{-1}_{m, l} \left[ U^{\infty}(z^l_m, -\theta_{i},\theta_{j})-V^{\infty}(-\theta_{i},
\theta_{j})\right]=-V^t(z^l_m, \theta_{i})V^t(z^l_m, \theta_{j})+O(a),
\end{equation}
$\; \mbox{ for } i,\; j=1, 2, ...,N_0; \;~ l=1, 2$. \footnote{We recall that $C_{m, l}=-4\pi a^{1-h}(1-a^h)$, see (\ref{actual-capacitances}).} 
As the two points $z^1_m$ and $z^2_m$ are close then the sign before the vectors above is the same for $l=1$ and $l=2$.
\bigskip

\item Knowing the vectors of total fields $(V^t(z^l_m, \theta_j))^{N_0}_{j=1}$, for $l=1,2$, (or $-(V^t(z^l_m, \theta_j))^{N_0}_{j=1}$, for $l=1,2$) the values of the Green function $G(z^1_m,z^2_m)$ can be recovered as follows. We have 
\begin{equation}\label{Green-function-1}
G(z^1_m, z^2_m):=(\tilde{\textbf{G}})_{1,2}\;, 
\end{equation}
where the matrix $\tilde{\textbf{G}}$ is computable as
 \begin{equation}\label{Green-function-2}
(\textbf{V}\textbf{V}^\top)^{-1}\textbf{C}^{-2} \textbf{V}\left[\mathbf{W}^{\infty} -\mathbf{V}^{\infty} \right]\textbf{V}^\top(\textbf{V}\textbf{V}^\top)^{-1}
\textcolor{black}{+} \textbf{C}^{-1}=\tilde{\textbf{G}}+\;~ O\left(
 a^{1-h-2t}+a^{h}\right)
\end{equation}
with the matrix $\textbf{V}$ constructed from (\ref{total-fields}) via 
$(\textbf{V})_{l,j}:=V^t(z^l_m, \theta_j)$, $l=1,2$ and 
$j=1, 2, ..., N_0$ and $\textbf{C}:=\diag\{C_{m,l}\}$. 
\footnote{ The matrix $\diag\{C_{m,l}\}$ is diagonal and the diagonal values are defined
by the vector $\{C_{m,1}, C_{m,2}\}$. Since $C_{m,l}=-4\pi a^{1-h}(1-a^h)$ as the shapes of the inclusions are the same,
i.e. balls of radius $a$, we have $\diag\{C_{m,l}\}=-4\pi a^{1-h}(1-a^h) I_{2}$ where 
$I_{2}$ is the identity matrix in $\mathbb{R}^{2}$.} The matrices $\mathbf{W}^{\infty}$ and $\mathbf{V}^{\infty}$ are defined as $(\mathbf{W}^{\infty})_{i, j}:=W^{\infty}(z^l_m;-\theta_i,\theta_j)$ and
$(\mathbf{V}^{\infty})_{i, j}:=V^{\infty}(-\theta_i,\theta_j)$, for $i, j=1, ...,N_0$. 
\bigskip

\item If, in addition, $n$ is of class $C^{\alpha},\; \alpha \in (0,1],$ 
then\footnote{We can also take $\alpha =0$, i.e. $n$ is only continuous, and obtain 
(\ref{formulas-for-n}) with an appropriate 
change in the error term.} from the values of the Green function 
$G(z^1_m,z^2_m)$, we reconstruct the values of the index of 
refraction $n(z^l_m)$ as follows:
\begin{equation}\label{formulas-for-n}
n(z^l_m)=\left[\frac{4\pi}{i\kappa} G(z^2_m, z^1_m)-\frac{1}{i\kappa\vert z^2_m-z^1_m\vert} \right ] +O(a^t),\;~ t>0.
\end{equation}
\end{enumerate}
\end{theorem}
\bigskip

\begin{remark}\label{Remark} \textcolor{black}{We make the following observations}:
\begin{enumerate}
\item We observe that due to the error $O(a^t)$, for $m=1,.2,...,M$, we do not distinguish 
between $n(z^1_m)$ and $n(z^2_m)$ since $\vert z^1_m-z^2_m\vert \sim a^t$, as $a\rightarrow 0$ 
and hence $\vert n(z^1_m)-n(z^2_m)\vert =O(a^t)$ (in the best case $\alpha=1$). 
Hence these formulas provide the values of $n(z_m)$, for $m=1,.2,...,M$, where $z_m:=z^l_m$
for either $l=1$ or $l=2$.     

\item  From the formulas (\ref{total-fields}),\; (\ref{Green-function-2}) and (\ref{formulas-for-n}), we see that the error in reconstructing $n$ is of the order $a^{\min\{t, h, 1-h-2t\}}$ and hence we need $h>0$ and $0<t<\frac{1-h}{2}$. The ratio
between the radius $a$ and the minimum distance $d\approx a^t$ is of the order $\frac{a}{d}\approx 
a^{1-t}\; <<1$ as $a<<1$. This rate becomes significantly smaller when $t$ approaches 
$0$. This means that the small perturbations do not need to be very close to each other to apply the reconstruction formulas
in Theorem \ref{Main-results}. This makes sense and can be useful in practice since, due to the 
resolution limit, it is delicate to use focusing waves to create close small 
deformations. The price to pay is that we will have less accuracy in the reconstruction, see the 
formula (\ref{formulas-for-n}) for instance. 
Hence, one needs to find a compromise between the desired accuracy of the reconstruction and the permissible closeness of the small 
perturbations created by focusing waves. 
For instance taking $t=h=\frac{1}{4}$, the error in the reconstruction is of the order $a^{\frac{1}{4}}$ and the ratio $\frac{a}{d}$ is of the order $a^{\frac{3}{4}}$, as $a<<1$ .   

\item The importance of the parameter $t$, describing the closeness of the small inclusions, 
appears in the last step (i.e. the point 3. in Theorem \ref{Main-results}) to compute the values of $n$ using the singularity analysis
of \textcolor{black}{$G$}. Indeed, in (\ref{Green-function-2}), we can even take  $t=0$ with $h\in (0,\; 1)$. 
If we need to choose $t$ very small (for practical or experimental reasons), then we can also use other methods to compute $n(x)$
 from $G(x, z)$. For instance,
from the Lippmann-Schwinger equation $G(x,z)=\Phi(x, z)+\kappa^2 \int_{\Omega}(n^2-1)(t)
G(t, z)\Phi(t, x)dt,\; x, z \in \Omega$, we can compute $(n^2-1)(x)
G(x, z)$ by inverting an integral equation of the first kind having as a kernel $\Phi(x, z)$.
Knowing $G(x, z)$ and $(n^2-1)(x)G(x, z)$ for a sample of points $x, z \in \Omega$, 
we can reconstruct $n$ from the formula $(n^2-1)(x)=\frac{(n^2-1)(x)
G(x, z)}{G(x, z)}$. The denominator $G(\cdot, z)$ is not vanishing if we choose appropriately
 different values of the source points $z$ in $\Omega$.
 
 \item In this work, we use localized perturbations modeled by small inclusions with impedance boundary conditions. 
 However, the natural localized perturbations should occur on the coefficients of the model \cite{A-B-C-T-F:2008}, i.e. the index of refraction $n$ in our model. 
In this respect, using surface impedance type inclusions might be 
more attractive in the academic level than in practice unless we accept and use the following argument. In case where there is a large contrast between the medium and the background, and according to 
the Leontovich approximation, the contrast can be replaced by the surface impedance boundary condition. 
This means that we should use localized perturbations creating quite high contrasts.
\end{enumerate}
\end{remark}

\section{Proof of Theorem \ref{Main-results}}
\subsection{Step 1: Use of  single inclusion to extract the field $\pm (V^t(z^1_m, \theta_j), V^t(z^2_m, \theta_j)),\; z^1_m, z^2_m \in \Omega$}
We create one single inclusion located at a selected point $z^l_m$ in $\Omega$. This means that we take $M=1$ in the asymptotic expansion (\ref{x oustdie1 D_m farmain-2}), 
i.e.
$
U^\infty(z^l_m, \hat{x},\theta)-V^\infty(\hat{x},\theta)=V^t(z_m^l, -\hat{x})\mathbb{Q}+
O\left(
 a^{2-h}\right)
$
and $\mathbb{Q}=-C_{m, l}V^t(z_m^l, \theta)$. Hence
\begin{equation}\label{M-1}
U^\infty(z^l_m; -\theta_i,\theta_j)-V^\infty(-\theta_i,\theta_j)=-C_{m, l} V^t(z^l_m, \theta_i)V^t(z^l_m, \theta_j)+\;
O\left(
 a^{2-h}\right).
 \end{equation}
\bigskip
Using the backscattered data $U^\infty(z^l_m,-\theta_j,\theta_j)$ and $V^\infty(-\theta_j,\theta_j)$, and since 
$C_{m,l}:=-4\pi a^{1-h}(1-a^h)$ is known, from (\ref{M-1}) we derive the formula
\begin{equation}\label{one-inclusion}
C^{-1}_{m,l} \left [U^\infty(z^l_m; -\theta_j,\theta_j)-V^\infty(-\theta_j,\theta_j)\right]=-(V^t(z^l_m, \theta_j))^2+\;
O\left(
 a\right).
\end{equation}
Hence we can construct either $V^t(z^l_m, \theta_j)$ or $-V^t(z^l_m, \theta_j)$ since the complex argument will be computed to an additive $k\pi$ period. From (\ref{M-1}), we see that we can compute  $V^t(z^l_m, \theta_i)$, $i\neq j$, with same additive period. Then, we can construct either the vector $(V^t(z^l_m, \theta_j))^{N_0}_{j=1}$ or  the vector $-(V^t(z^l_m, \theta_j))^{N_0}_{j=1}$. 

Let us now show why if $z^1_m$ is close to $z^2_m$, then this sign should be the same for $l=1$ and $l=2$. 

We know that the total fields $\textcolor{black}{V}^t(\cdot, \theta_j), j:=1,..., N,$ are solutions to the Lippmann-Schwinger equations $
V^t(x, \theta_j)+\kappa^2\int_{\Omega}(n^2(z)-1)\Phi(x, z)V(z, \theta_j)dz=
u^i(x, \theta_j),\; x \in \mathbb{R}^3.$ Based on this equation and following the arguments in Lemma 2.2 of \cite{BA-DPC-KMSM:JMAA:2015}, we show that $V(x, \theta)$ is in $W^{1, \infty}(\Omega)$ with the bound 
\begin{equation}\label{regularity-L-S}
 \Vert V^t(\cdot, \theta)\Vert_{W^{1,\infty}{(\Omega)}} \leq C \Vert u^i(\cdot, \theta)\Vert_{H^2(\tilde{\Omega})} 
\end{equation}
where $C$ depends only on the upper bound of $n$ and $\tilde{\Omega}$, here $\Omega \subset \subset\tilde{\Omega}$. As $\theta$ varies in $\mathbb{S}^2$, the constant $ \Vert u^i(\cdot, \theta)\Vert_{H^2(\tilde{\Omega})}$ is uniformly bounded. We set $\tilde C:=C\sup_{\theta \in \mathbb{S}^2}
\Vert u^i(\cdot, \theta)\Vert_{H^2(\tilde{\Omega})}$. Hence we deduce that
\begin{equation}\label{upper-bound}
\vert V^t(z^1_m, \theta)-V^t(z^2_m, \theta)\vert\; \leq \; \tilde{C}\; \vert z^1_m-z^2_m\vert.
\end{equation}
Hence if the points $z^1_m$ and $z^2_m$ are chosen such that $\vert z^1_m-z^2_m\vert$ is small enough, then $V^t(z^1_m, \theta)$ and $V^t(z^2_m, \theta)$ will have the same sign\footnote{For $m:=1, 2, ..., M$ and $l:=1,2$ fixed, the value $V^t(z^l_m, \theta_j)$ can vanish for some $j$ but according to Lemma \ref{Matrix-V-inver} it cannot be true for ever $j:=1,..., N_0 $ where $N_0$ is large.}.


\subsection{Step 2: Use of multiple inclusions to extract the Green's function $G(z^1_m, z^2_m),\; z^1_m, z^2_m \in \Omega$}
In this case, we create two inclusions located close to each other, i.e. the two points $z^1_m$ and $z^2_m$, $m=1, ..., M$, 
in $\Omega$, are such that $\vert z^2_m-z^1_m \vert \sim a^t$, as $a \rightarrow 0$, with $t>0$. Let $m$ be fixed. 
\bigskip

From the asymptotic expansion (\ref{x oustdie1 D_m farmain-2}), we have 

\begin{equation}\label{M-finite}
W^\infty(z^l_m;\hat{x},\theta)-V^\infty(\hat{x},\theta)=\sum_{l=1}^{2}V^t(z^l_m, -\hat{x})\mathbb{Q}_m^l+
O\left(
 M\mathcal{C}a\right),
 \end{equation}
where 

\begin{eqnarray}\label{Q-M-finite}
 \mathbb{Q}_m^l +\sum_{\substack{l^{\prime}=1 \\ l^{\prime}\neq l}}^{2}C_{m,l} G(z^l_m,z_m^{l^{\prime}}) \mathbb{Q}_m^{l^{\prime}}&=&-C_{m,l}V^{t}(z^l_m, \theta).
\end{eqnarray}
Hence, we rewrite (\ref{M-finite}) as 
\begin{equation}\label{M-finite-1}
W^\infty(\hat{x},\theta)-V^\infty(\hat{x},\theta)=\mathbf{V}^\top(-\hat{x})\mathbf{B}^{-1}\mathbf{V}(\theta)
+\; O\left(
 M\mathcal{C}a\right),
 \end{equation}
where we use the vector $\mathbf{V}(\theta):=\left[V^t(z^1_m, \theta), V^t(z^2_m, \theta)\right]^\top$ 
and the $2 \times 2$ matrix $\mathbf{B}$ whose components are $\mathbf{B}_{l, {l^{\prime}}}:= -G(z^l_m,z^{l^{\prime}}_m)$ for $l \neq {l^{\prime}}$ and 
$\mathbf{B}_{l, l}:=-C_{m,l}^{-1}$ for $l=1, 2$. 
\bigskip

We use a number $N$ of directions of incidences (and hence of propagation directions) larger than 
the number of sampling points $z^1_m, z^2_m$, i.e. $2$ directions, where we want to evaluate the index of refraction. 
With this number at hand, the matrix 
\begin{equation}\label{Matrix-V}
\mathbf{V}:=\left[\mathbf{V}(\theta_1), \mathbf{V}(\theta_2),..., \mathbf{V}(\theta_{N})\right]
\end{equation}
has a full rank and hence $\mathbf{V}\mathbf{V}^\top$ is invertible. Precisely, we have the following lemma.
\begin{lemma}\label{Matrix-V-inver}
There exists $N_0 \in \mathbb{N}$ such that if the number $N$ of incident directions $\theta_j,~ j=1,..., N$ is larger than $N_0$, 
i.e. $N \geq N_0$, then the vectors $\mathbf{V}^t_n(\theta_j):=\left[V^t(z^1_m, \theta_j), V^t(z^2_m, \theta_j\right]^\top$, $j=1, ..., N$ 
are linearly independent.
\end{lemma}

\begin{proof}
We observe, by the mixed reciprocity relations, that $\textcolor{black}{V}^t(z^l_m, \theta_j)=\textcolor{black}{G^\infty}(-\theta_j, z^l_m)$, where ${G^\infty}(-\theta_j, z^l_m)$ is the far-field 
in the direction $-\theta_j$ of the point source $G(x, z^l_m)$ supported at the source point $z^l_m$. We recall that the point source $G(x, z^l_m)$ 
is the Green function of the model (\ref{acimpoenetrable-2}). With this relation at hand, the proof of the lemma follows the same arguments in \cite{K-G} 
where it is done for the vectors $(e^{-i\kappa \theta_j\cdot z^l_m})_{l=1}^{2}$ and $j=1,..., N$, with $N$ large enough. We omit the details. 
Observe that $V^t(z^l_m, \theta_j)=e^{i\kappa \theta_j\cdot z^l_m}$ if the index of refraction $n$ is equal to unity in $\mathbb{R}^3$.
\end{proof}

\begin{remark}
In practice, one may take $N=N_0=2$. This can be easily justified for the case when the contrast of the index of refraction, $(n-1)$, is small. To see it we set, keeping the same notation as above, $\mathbf{V}:=\left[\mathbf{V}(\theta_1), \mathbf{V}(\theta_2)\right]$ where $\mathbf{V}(\theta_j):=\left[V^t(z^1_m, \theta), V^t(z^2_m, \theta_j)\right]^\top$, $j=1, 2$. Correspondingly, we set
$\mathbf{V}_1(\theta_j):=\left[e^{i\kappa z_m^1\cdot \theta_j}, e^{i\kappa z^2_m\cdot \theta_j}\right]^\top$, $j=1,2$ and $\mathbf{V}_1:=\left[\mathbf{V}_1(\theta_1), \mathbf{V}_1(\theta_2))\right]$. A simple computation shows that the determinant of $\mathbf{V}_1$ is $e^{i\kappa(\theta_1+\theta_2){\cdot}z^1_m}[e^{i \kappa (z^2_m-z^1_m)\cdot\theta_2}-e^{i \kappa (z^2_m-z^1_m)\cdot\theta_2}]\neq 0$ if $\theta_1\neq \theta_2$. Now, we observe, by the Lippmann-Schwinger equation for instance, that $\mathbf{V}=\mathbf{V}_1+ O(\vert n-1\vert)$. Hence if the contrast $n-1$ is small enough then the determinant of $\mathbf{V}$ is not vanishing.  

\end{remark}

Hence, by measuring the data $\mathbf{W}^\infty$ and $\mathbf{V}^\infty$  using the formulas
\begin{equation}\label{matrix-data-1}
(\mathbf{W}^\infty)_{i,j}:=(W^\infty(z^l_m;-\theta_i,\theta_j)), \; i, j=1, ...,N_0
\end{equation} 
and
\begin{equation}\label{matrix-data-2}
(\mathbf{V}^\infty)_{i,j}:=(V^\infty(-\theta_i,\theta_j)), \; i, j=1, ...,N_0,
\end{equation} 
we obtain the formula 
\begin{equation}\label{Matrix-B-Q}
\mathbf{W}^{\infty}-\mathbf{V}^{\infty}=\mathbf{V}^\top \textbf{B}^{-1}\mathbf{V}\;
+ O\left(a^{2-h}\right).
\end{equation}
Observe that $\textbf{B}^{-1}=-\left ( I+\textbf{C} \tilde{\textbf{G}}\right)^{-1}\textbf{C}$ where $\textbf{C}:=diag\left(C_{m,l}\right)$ and 
$\left(\tilde{\textbf{G}} \right)_{l,l}:=0, l=1, 2$ and $\left(\tilde{\textbf{G}} \right)_{l,l^\prime}:=G(z_m^l,z^{l^\prime}_m), l \neq l^\prime$. 
Hence we can write 
\begin{equation}\label{First-order-appr}
\textbf{B}^{-1}=-\textbf{C} +\textbf{C} \tilde{\textbf{G}} \textbf{C} +O(\max_{l}\{C_{m,l}\}\Vert \textbf{C} \textbf{G}_\kappa\Vert^2_2)
\end{equation} 
  where $\Vert \Vert_2$ is the $l_2$-norm for matrices. 
 
 We recall that we have chosen the points $z^1_m$ and $z^2_m$, in $\Omega$, such that $\vert z^2_m-z^1_m \vert \sim a^t$ with $t\in (0, 1)$, as $a \rightarrow 0$. With such a choice, we see that

\begin{equation}\label{First-order-appr-1}
\textbf{B}^{-1}=-\textbf{C} +\textbf{C} \tilde{\textbf{G}}\textbf{C}  +O\left( a^{3-3h-2t} \right) 
\end{equation}
as $O(\max_{l}\{C_{m,l}\}\Vert \textbf{C} \tilde{\textbf{G}}\Vert^2_2)=O(a^{3-3h-2t})$.

Finally, from the matrix $\textbf{B}$ we derive the values of the Green function at the 
sample points $z_m^l,\:l=1,2$, in $\Omega$:
\begin{equation}\label{G}
G(z^1_m,z^2_m).
\end{equation}

Indeed, combining (\ref{Matrix-B-Q}) and (\ref{First-order-appr-1}), we deduce that 
\begin{equation}\label{approximation}
\mathbf{W}^{\infty}-\mathbf{V}^{\infty}=-\mathbf{V}^\top \textbf{C}\mathbf{V}\; 
+ \mathbf{V}^\top \textbf{C} \tilde{\textbf{G}} \textbf{C} \mathbf{V}\;+~ O\left(
 a^{3-3h-2t}+a^{2-h}\right)
\end{equation}
where $\mathbf{V}^\top \textbf{C}\mathbf{V}\sim a^{1-h}$ and $ \mathbf{V}^\top 
\textbf{C} \tilde{\textbf{G}} \textbf{C} \mathbf{V} \sim a^{2-2h-t}$. Remember that from step. $1$, we compute $\pm \mathbf{V}$ (i.e. we do know the correct sign), but as in (\ref{approximation}) we multiply the two sides by $\mathbf{V}$ and $\mathbf{V}^\top$, then the missed sign of $\mathbf{V}$ is irrelevant.
\bigskip
 
The approximation (\ref{approximation}) will make sense, i.e. the error term is dominated by 
$-\mathbf{V}^\top \textbf{C}\mathbf{V}\; +  \mathbf{V}^\top \textbf{C} \tilde{\textbf{G}} \textbf{C} \mathbf{V}$, 
if $3-3h-2t>2-2h -t$ and $2-h>2-2h -t$, i.e. $0<t+h<1$. In addition, 
$\mathbf{V}^\top \textbf{C}  \tilde{\textbf{G}}\textbf{C}  
\mathbf{V} =O(\mathbf{V}^\top \textbf{C}\mathbf{V})$ if $2-2h -t\geq 1-h$, i.e. $t+h\leq 1$. 
Hence the approximation (\ref{approximation}) makes sense if $0<t+h <1$. This last condition is satisfied since 
we assumed that $0<h$ and $h+2t<1$.

From Step 1, we showed already how to extract the matrices $\pm \mathbf{V}$, hence we can 
compute $\mathbf{W}^{\infty}-\mathbf{V}^{\infty}-\mathbf{V}^\top \textbf{C}\mathbf{V}\;$. 
Since the matrix $\textbf{C}$ is diagonal, with non zero entries, and the matrix 
$\textbf{V}\textbf{V}^\top$ is invertible, 
as $\textbf{V}$ is full rank matrix, then we can recover the matrix 
$\tilde{\textbf{G}}$ as
\begin{equation}\label{reconvering-G-kappa}
\textbf{C}^{-1}(\textbf{V}\textbf{V}^\top)^{-1} \textbf{V}\left[\mathbf{W}^{\infty}-\mathbf{V}^{\infty}\right]\textbf{V}^\top(\textbf{V}\textbf{V}^\top)^{-1}\textbf{C}^{-1}
+ \textbf{C}^{-1}=\tilde{\textbf{G}}+\;~ O(a^{2h-2})\;O\left(
 a^{3-3h-2t}+a^{2-h}\right)
\end{equation}
or 
\begin{equation}\label{reconvering-G-kappa-1}
\textbf{C}^{-1}(\textbf{V}\textbf{V}^\top)^{-1}\textbf{V}\left[\mathbf{W}^{\infty} -\mathbf{V}^{\infty}\right]\textbf{V}^\top(\textbf{V}\textbf{V}^\top)^{-1}\textbf{C}^{-1}
+ \textbf{C}^{-1}=\tilde{\textbf{G}}+\;~ O\left(
 a^{1-h-2t}+a^{h}\right).
\end{equation}
We see that $a^{h}=o(\tilde{\textbf{G}})$ if $h>-t$,i.e. $t+h>0$ and $a^{1-h-2t}=o(\tilde{\textbf{G}})$ 
if $1-h-2t>-t$, i.e. $t+h<1$. Hence (\ref{reconvering-G-kappa-1}) makes sense if $0< t+h <1$. 
As mentioned above, this last condition is satisfied since we assumed that $0<h$ and $h+2t<1$.

Finally, we set
\begin{equation}\label{dominant-matrix}
G(z^1_m, z^2_m):=(\tilde{\textbf{G}})_{1, 2},\;~ 
\end{equation}
as the reconstructed values of the Green function $G$ at the sampling points \textcolor{black}{$(z^1_m, z^2_m),\; m=1, 2,...,M$}.

\subsection{Step 3: Extraction of the index of refraction from the Green function}
We start with the following lemma

\begin{lemma} We have the asymptotic expansion:
\begin{equation}\label{asymp-G}
G (x, z_j)= \frac{e^{i\kappa n(z_j)\vert x-z_j\vert}}{4\pi \vert x-z_j\vert}+O(\vert x-z_j\vert),\; \mbox{ as } \vert x-z_j\vert \rightarrow 0.
\end{equation}
\end{lemma}

\begin{proof}
We know that $(\Delta +\kappa^2 n^2(x))G=-\delta,$ in $\mathbb{R}^3$ and 
$\Phi_{ j}(x, z):=\frac{e^{i\kappa n(z_j)\vert x-z\vert}}{4\pi \vert x-z\vert}$ 
satisfies $(\Delta +\kappa^2 n^2(z_j))\Phi=-\delta,$ in $\mathbb{R}^3$, 
where both $G$ and $\Phi_{ j}$ satisfy the Sommerfeld radiation condition. 
Then $H_{ j}(x, z):=(G-\Phi_{ j})(x, z)$ satisfies the same radiation 
condition and 
\begin{equation}\label{difference}
(\Delta +\kappa^2 n^2(z_j))H_{ j}=\kappa^2 (n^2(z_j)-n^2(x))G, \mbox{ in } \mathbb{R}^3. 
\end{equation}
Multiplying both sides of (\ref{difference}) by $\Phi_{ j}$ and integrating over 
a bounded and smooth domain $B$ we obtain:

\begin{eqnarray}\label{representation-1}
H_{ j}(x, z)&=& -\kappa^2\int_{B}(n^2(z_j)-n^2(t))\Phi_{ j}(t, x)G(t, z)dt \\
&&-\int_{\partial B} H_{ j}(t, z)\partial_{\nu(t)}\Phi_{ j}(t, x)ds(t) 
+\int_{\partial B} \partial_{\nu(t)} H_{ j}(t, z)\Phi_{ j}(t, x)ds(t),\nonumber
\end{eqnarray}
and then 
\begin{eqnarray}\label{representation-2}
\nabla_{x} H_{ j}(x, z_j)&=& -\kappa^2\int_{B}(n^2(z_j)-n^2(t))\Phi_{ j}(t, x)
\nabla_{x}G(t, z_j)dt \\
&&-\int_{\partial B} H_{ j}(t, z_j)\nabla_{x}\partial_{\nu(t)}\Phi_{ j}(t, x)ds(t)  
+ \int_{\partial B} \partial_{\nu(t)} H_{ j}(t, z_j)\nabla_x \Phi_{ j}(t, x)ds(t).\nonumber
\end{eqnarray}
Here, we choose $B$ such that $\Omega \subset \subset B$.  Since $(n^2(z_j)-n^2(t))
\nabla G(t, z_j)=O(\vert t-z_j\vert^{-2 +\alpha}), 
t \in B$, as the singularity of $\nabla_{x}G(t, -z_j)$ is of the order 
$\vert t-z_j\vert^{-2}$, and $n^2(z_j)-n^2(t) =O(\vert t-z_j\vert^\alpha)$, since $n$ is of
class $C^\alpha$, we deduce that $\int_{B}(n^2(z_j)-n^2(t))\Phi_{ j}(t, x)
\nabla_{x}G(t, z_j)dt =O(\int_B \vert t-z_j\vert^{-2+\alpha}\vert t-x\vert^{-1}dt)
=O(1),\; x \in B$ since $\alpha >0$, see (\cite{Val}, Lemma 4.1), for instance, for the such integrals o
f of singular functions. The two last integrals appearing in (\ref{representation-2}) are of the order $O(1)$ for $x \in \Omega \subset \subset B$.

Now, let $B_{z_j, r}$ be the ball of center $z_j$ and radius $r,\; r <<1$. 
We use (\ref{representation-1}) applied to $B_{z_j, r}$ instead of $B$:
\begin{eqnarray}\label{representation-3}
H_{ j}(x, z)&=&-\kappa^2\int_{B_{z_j, r}}(n^2(z_j)-n^2(t))\Phi_{ j}(t, x)G(t, z)dt \\
&&-\int_{\partial B_{z_j, r}} H_{ j}(t, z)\partial_{\nu(t)}\Phi_{ j}(t, x)ds(t) 
+  \int_{\partial B_{z_j, r}} \partial_{\nu(t)} H_{ j}(t, z)\Phi_{ j}(t, x)ds(t)\nonumber
\end{eqnarray}
for $x \in B_{z_j, r}$. We see that
\begin{eqnarray*}
\int_{B_{z_j, r}}(n^2(z_j)-n^2(t))\Phi(t, x)G(t, z_j)dt&=& 
O\Big ( \int_{B_{z_j, r}} \vert t-z_j \vert^{-1+\alpha} \vert t-x\vert^{-1} ds(t)\Big)\,
,\mbox{ for } x \in B_{z_j,  r}
\end{eqnarray*}
which we can rewrite as 
\begin{eqnarray*}
\int_{B_{z_j, r}}(n^2(z_j)-n^2(t))\Phi(t, x)G(t, z_j)dt&=& 
O\Big ( (\int_{B_{z_j, r}} \vert t-z_j \vert^{-2+2\alpha} ds(t))^{\frac{1}{2}}
(\int_{B_{x, 2r}} \vert t-x \vert^{-2} ds(t))^{\frac{1}{2}}\Big)\,
,
\end{eqnarray*}
for $x \in B_{z_j, r}$. Hence
\begin{eqnarray*}
\int_{B_{z_j, r}}(n^2(z_j)-n^2(t))\Phi(t, x)G(t, z_j)dt&=& 
O(r^{1+\alpha})\,
, \mbox{ for } x \in B_{z_j, r}.
\end{eqnarray*}

Since $\nabla H_{ j}(t, z_j)$ is bounded for $t\in \Omega$, then
\begin{eqnarray*}
\int_{\partial B_{z_j, r}} \partial_{\nu(t)} H_{ j}(t, z_j)\Phi_{ j}(t, x)ds(t)&=&O\Big(\int_{\partial B_{z_j, r}}  \vert \Phi_{ j}(x, t)\vert ds(t)\Big)=O(r)
\end{eqnarray*}
for $x \in \Omega$. Hence, we can write (\ref{representation-3}) as 
\begin{equation}\label{representation-4}
H_{ j}(x, z_j) +\int_{\partial B_{z_j, r}} H_{ j}(t, z_j)\partial_{\nu(t)}
\Phi_{ j}(t, x)ds(t) = O(r), \mbox{ for } x \in B_{z_j, r}.
\end{equation}
Taking the trace in (\ref{representation-4}) to $\partial B_{z_j, r}$ and using the trace 
of the double layer potential $\int_{\partial B_{z_j, r}} H_{ j}(t, z_j)\partial_{\nu(t)}
\Phi_{ j}(t, x)ds(t)=K_r(H_{ j}(\cdot, z_j))(x) -\frac{1}{2} H_{ j}(x, z_j),
\;~ x \in \partial B_{z_j, r}$, where $K_r$ denotes the double layer operator, we derive the equation

\begin{equation}\label{representation-5}
\left (\frac{1}{2} +K_r \right )(H_{ j}(\cdot, z_j))= O(r), \mbox{ for } x \in \partial B_{z_j, r}.
\end{equation}
Then 
$$
\Vert H_{ j}(\cdot, z_j)\Vert_{L_2(\partial B_{z_j, r})}=O\left (r\; \Vert(\frac{1}{2} +K_r )^{-1}\Vert_{\mathcal{L}( L_2(\partial B_{z_j, r}),\; L_2(\partial B_{z_j, r}))}\right).
$$
We have the following estimate from \cite{DPC-SM13}:
$$
\Vert (\frac{1}{2} +K_r )^{-1})\Vert_{\mathcal{L}( L_2(\partial B_{z_j, r}),\; 
L_2(\partial B_{z_j, r}))} \leq \Vert(\frac{1}{2} +K_1 )^{-1})\Vert_{\mathcal{L}
( L_2(\partial B),\; L_2(\partial B))}
$$
where $B:=B_{O, 1}$. Hence
\begin{equation}\label{abs-W}
\Vert H_{ j}(\cdot, z_j)\Vert_{L_2(\partial B_{z_j, r})}=O(r).
\end{equation}
Going back to (\ref{representation-4}), we estimate
$$
\vert H_{ j}(x, z_j)\vert \leq \int_{\partial B_{z_j, r}} \vert H_{ j}(t, z_j)\vert \vert \partial_{\nu(t)}\Phi_{ j}(x, t) \vert ds(t) +O(r)
$$
and then 
$$
\vert H_{ j}(x, z_j)\vert \leq \Vert H_{ j}(\cdot, z_j)\Vert_{L^2(\partial B_{z_j, r})} \vert \Vert \partial_{\nu(t)}\Phi_{ j}(\cdot, t)\Vert_{L^2(\partial B_{z_j, r})} +O(r).
$$
From (\ref{abs-W}) and the fact that 
$\Vert \partial_{\nu(t)}\Phi_{ j}(\cdot, t)\Vert_{L^2(\partial B_{z_j, r})} =
O(\int_{\partial B(z_j, r)}\vert t-z_j\vert^{-2}dt)^{\frac{1}{2}}=O(1)$, we derive the estimate
$$
H_{ j}(x, z_j)=O(r)
$$
for $x \in \overline{B_{z_j, r}}$ and in particular for $x$ such that $\vert x-z_j\vert =r$.  
Hence, from the definition of $H_{ j}(x, j)$, we deduce that
$$
G (x, z_j)= \frac{e^{i\kappa n(z_j)\vert x-z_j\vert}}{4\pi \vert x-z_j\vert}+O(\vert x-z_j\vert),\; \mbox{ as } \vert x-z_j\vert \rightarrow 0.
$$
\end{proof}

From (\ref{asymp-G}), we deduce that,
\begin{equation}\label{Final-formula}
G(x, z_m)=\frac{1}{4\pi \vert x-z_m\vert}+ i\;\kappa\; n(z_m) \; + O(\vert x-z_m\vert),\;~~ \mbox{ as } \vert x-z_m\vert \rightarrow 0
\end{equation}
and then
\begin{equation}\label{Final-formula-1}
n(z_m)=\lim_{x\rightarrow z_m}\left[\frac{1}{i\kappa} G(x, z_m)-\frac{1}{i\kappa\;4\pi\vert x-z_m\vert} \right ].
\end{equation} 
The formula (\ref{Final-formula-1}) can be used for $x:=z_m^2$ and $z_m:=z^1_m$, where $z_m^1$ and $z^2_m$, \textcolor{black}{$m=1,2,..., M$} are the sampling points added to the small inclusions,
\begin{equation}\label{Final-formula-2}
n(z_m)=\left[\frac{1}{i\kappa} G(z^2_m, z^1_m)-\frac{1}{i\kappa\;4\pi\vert z^2_m-z^1_m\vert} \right ] +O(a^t),\;~ t>0.
\end{equation}

\section{The stability issue}

The object of this section is to analyze how the formulas are accurate when the measured far-fields are contaminated with noise and when the locations of the small perturbations
$z^l_m, m:=1, ..., M \mbox{ and } l=1,2$, are slightly shifted when passing from step 2 to step 3 in collecting the measured far-fields. Precisely, we will derive the regimes on the amplitude of the noise and the length of the shifts under which our formulas are still applicable.

We recall the main formulas \eqref{reconvering-G-kappa-1}, \eqref{dominant-matrix} and \eqref{Final-formula-2};

\begin{equation}\label{reconvering-G-kappa-1--}
\textbf{C}^{-1}(\textbf{V}\textbf{V}^\top)^{-1}\textbf{V}\left[\mathbf{W}^{\infty} -\mathbf{V}^{\infty}\right]\textbf{V}^\top(\textbf{V}\textbf{V}^\top)^{-1}\textbf{C}^{-1}
+ \textbf{C}^{-1}=\tilde{\textbf{G}}+\;~ O\left(
 a^{1-h-2t}+a^{h}\right).
\end{equation} 

\begin{equation}\label{dominant-matrix--}
G(z^1_m, z^2_m):=(\tilde{\textbf{G}})_{1,2},
\end{equation}

and
\begin{equation}\label{Final-formula-2-r}
n(z_m)=\left[\frac{1}{i\kappa} G(z^2_m, z^1_m)-\frac{1}{i\kappa\;4\pi\vert z^2_m-z^1_m\vert} \right ] +O(a^t),\;~ t>0.
\end{equation}

Combining them, we obtain:

\begin{eqnarray}\label{true-formula}
i\kappa\, n(z_m)&=&
\left[\textbf{C}^{-1}(\textbf{V}\textbf{V}^\top)^{-1}\textbf{V}\left[\mathbf{W}^{\infty} -\mathbf{V}^{\infty}\right]\textbf{V}^\top
(\textbf{V}\textbf{V}^\top)^{-1}\textbf{C}^{-1}
+ \textbf{C}^{-1}\right]_{_{1,2}} -\frac{1}{\;4\pi\vert z^2_m-z^1_m\vert}\quad\\
&&+O\left(
 a^{1-h-2t}+a^{h}\right)+O\left(a^{t}\right).\nonumber 
\end{eqnarray}

We recall also that \textcolor{black}{$\mathbf{V}=[\mathbf{V}(\theta_1),\mathbf{V}(\theta_2),\cdots,\mathbf{V}(\theta_N)]$} where  $\mathbf{V}(\theta):=\left[V^t(z^1_m, \theta), V^t(z^2_m, \theta)\right]^\top$ and $V^t(z^l_m,\theta_j)$ can be computed, see \eqref{M-1} and \eqref{one-inclusion}, as
\begin{eqnarray}\label{one-inclusion-stab}
 C^{-1}_{m,l} \left [U^\infty(z^l_m; -\theta_i,\theta_j)-V^\infty(-\theta_i,\theta_j)\right]&=&-V^t(z^l_m, \theta_i)V^t(z^l_m, \theta_j)+\;
O\left(a\right).
\end{eqnarray}

We base our remaining analysis on the two formulas (\ref{true-formula}) and (\ref{one-inclusion-stab}).

We set $\tilde{z}^l_m$ to be the possibly shifted positions of the small perturbations while moving from the second step to third step of collecting the far-fields, see Section \ref{Sec:extractform}. We set also 

$$\eta:=\max_{m:=1, 2, ..., M;\, l:=1,2}\vert \tilde{z}^l_m-z^l_m \vert$$
and
$$
\tilde{d}:=\min_{\substack{m\neq m^{\prime}\mbox{ (or) }l\neq l^{\prime}\\ m,m^{\prime}=1,\dots,M;\,l,l^{\prime}=1,2}}d(\tilde{D}_m^l, \tilde{D}_{m^{\prime}}^{l^{\prime}}) \approx a^{\tilde{t}}
$$

where $\tilde{D}_m^l$'s are the ball of center $\tilde{z}_m^l$'s and radius $a$ respectively. Let now $\delta_v, \delta_u$ and $\delta_w$ be the additive noises corresponding to $V^{\infty}, U^\infty$ and $W^\infty$ respectively. The actual measured data is 
$V_{\delta}^{\infty}, U_{\delta}^\infty$ and $W_{\delta, n}^\infty$ satisfying respectively
\begin{equation}
\vert V_{\delta}^{\infty}(-\theta_i, \theta_j)- V^{\infty}(-\theta_i, \theta_j)\vert \leq \delta_v, \; \vert U_{\delta}^{\infty}(z^l_m; -\theta_i, \theta_j)-U^{\infty}(z^l_m, -\theta_i, \theta_j) \vert \leq \delta_u$$
and $$ \vert W_{\delta}^{\infty}\textcolor{black}{(\tilde{z}^l_m; -\theta_i, \theta_j)} -W^{\infty}\textcolor{black}{(\tilde{z}^l_m; -\theta_i, \theta_j)}\vert \leq \delta_w
\end{equation}
uniformly in terms of $\theta_i, \theta_j \in \mathbb{S}^2$, $i, j=1,..., N_0$.  
Observe that since $
C^{-1}_{m,l} \left [U^\infty(z^l_m; -\theta_i,\theta_j)-V^\infty(-\theta_i,\theta_j)\right]=-V^t(z^l_m, \theta_i)V^t(z^l_m, \theta_j)+\;
O\left(a\right)\nonumber
$ 
then
\begin{eqnarray}\label{V-with-noisy-data}
 C^{-1}_{m,l} \left [U_{\delta}^\infty(z^l_m, -\theta_j,\theta_j)-V_{\delta, n}^\infty(-\theta_j,\theta_j)\right]=-V^t(z^l_m, \theta_i) V^t(z^l_m, \theta_j)+\;
O\left(a\right)+C^{-1}_{m,l} (\delta_u+\delta_v).
\end{eqnarray}

We set \textcolor{black}{$\tilde{\mathbf{V}}=[\tilde{\mathbf{V}}(\theta_1),\tilde{\mathbf{V}}(\theta_2),\cdots,\tilde{\mathbf{V}}(\theta_N)]$} where  $\tilde{\mathbf{V}}(\theta):=\left[V^t(\tilde{z}_m^1, \theta), V^t(\tilde{z}_m^2, \theta)\right]^\top$. Remember now that in the second step, the data are collected at the shifted points \textcolor{black}{$\tilde{z}^l_{m}, m=1, 2, ..., M$ and $l=1, 2$}.
Hence instead of (\ref{true-formula}), we have

\begin{eqnarray}\label{true-formula-shifted}
i\kappa n(\tilde{z}_m)&=&
\left[\textbf{C}^{-1}(\tilde{\textbf{V}}\tilde{\textbf{V}}^\top)^{-1}\tilde{\textbf{V}}
\left[\tilde{\mathbf{W}}^{\infty} -\mathbf{V}^{\infty}\right]\tilde{\textbf{V}}^\top(
\tilde{\textbf{V}}\tilde{\textbf{V}}^\top)^{-1}\textbf{C}^{-1}
+ \textbf{C}^{-1}\right]_{_{1,2}} -\frac{1}{\;4\pi\vert \tilde{z}^2_m-\tilde{z}^1_m\vert} \quad\\
&&+O\left(
 a^{1-h-2\tilde{t}}+a^{h}\right)+O\left(
 a^{\tilde{t}}\right)\nonumber
\end{eqnarray}
where we set $\tilde{\mathbf{W}}^{\infty}$ to be the matrix defined as 
$(\tilde{\mathbf{W}}^{\infty})_{i, j}:=W^{\infty}\textcolor{black}{(\tilde{z}^l_m; -\theta_i, \theta_j)}$ and correspondingly 
we set the matrix $\tilde{\mathbf{W}}_{\delta}^{\infty}$ defined as $(\tilde{\mathbf{W}}_{\delta}^{\infty})_{i, j}:=W_{\delta}^{\infty}\textcolor{black}{(\tilde{z}^l_m; -\theta_i, \theta_j)}$ . Here $\tilde{t}$ satisfies the conditions
\begin{equation}
0<h,\;0<\tilde{t},\; h+2\tilde{t}<1.
\end{equation}
\bigskip

From this formula, we get the corresponding one with noisy data $\mathbf{W}_{\delta}^{\infty}$ and $\mathbf{V}_{\delta}^{\infty}$

\begin{eqnarray}\label{true-formula-shifted-1}
i\kappa\, n(\tilde{z}_m)&=&
\left[\textbf{C}^{-1}(\tilde{\textbf{V}}\tilde{\textbf{V}}^\top)^{-1}\tilde{\textbf{V}}
\left[\tilde{\mathbf{W}}_{\delta}^{\infty} -\mathbf{V}_{\delta}^{\infty}\right]\tilde{\textbf{V}}^\top(
\tilde{\textbf{V}}\tilde{\textbf{V}}^\top)^{-1}\textbf{C}^{-1}
+ \textbf{C}^{-1}\right]_{_{1,2}} -\frac{1}{\;4\pi\vert \tilde{z}^2_m-\tilde{z}^1_m\vert} \qquad\\
&&+ (\delta_w+\delta_v)\vert \textbf{C}^{-1}(\tilde{\textbf{V}}\tilde{\textbf{V}}^\top)^{-1}
\tilde{\textbf{V}}\vert \vert \tilde{\textbf{V}} (\tilde{\textbf{V}}\tilde{\textbf{V}}^\top)^{-1}
\textbf{C}^{-1}\vert+\;~ O\left(
 a^{1-h-2\tilde{t}}+a^{h}\right)+O\left(
 a^{\tilde{t}}\right).\nonumber
\end{eqnarray}

Now, we need to replace the values of $\tilde{\textbf{V}}$ by the values of $\textbf{V}$, more precisely the computed values of $\textbf{V}$ from step 1 with noisy data. Let us set $
\textbf{V}_{\delta}$ the computed one from step 1. We deduce then from (\ref{upper-bound}) and (\ref{V-with-noisy-data}) that
\begin{equation}\label{noisy-shifted-values-V}
\vert \textbf{V}_{\delta} -\tilde{\textbf{V}} \vert =O(a+a^{h-1} (\delta_u+\delta_v)+\eta).
\end{equation} 

With this estimate at hand, the final formula is:
\begin{equation}\label{true-formula-shifted-2}
i\kappa\, n(\tilde{z}_m)=\\
\left[\textbf{C}^{-1}(\textbf{V}_{\delta}\textbf{V}_{\delta}^\top)^{-1}\textbf{V}_{\delta}
\left[\tilde{\mathbf{W}}_{\delta}^{\infty} -\mathbf{V}_{\delta}^{\infty}\right]\textbf{V}_{\delta}^\top(
\textbf{V}_{\delta}\textbf{V}_{\delta}^\top)^{-1}\textbf{C}^{-1}
+ \textbf{C}^{-1}\right]_{_{1,2}} -\frac{1}{\;4\pi\vert \tilde{z}^2_m-\tilde{z}^1_m\vert} 
\end{equation}
$$
+O(a^{2h-2})O(a+a^{h-1} (\delta_u+\delta_v)+\eta)\vert \tilde{\mathbf{W}}_{\delta}^{\infty} -\mathbf{V}_{\delta}^{\infty}\vert +
(\delta_w+\delta_v)\vert \textbf{C}^{-1}(\tilde{\textbf{V}}\tilde{\textbf{V}}^\top)^{-1}
\tilde{\textbf{V}}\vert \vert \tilde{\textbf{V}} (\tilde{\textbf{V}}\tilde{\textbf{V}}^\top)^{-1}
\textbf{C}^{-1}\vert 
$$
$$+ O\left(
 a^{1-h-2\tilde{t}}+a^{h}\right)+O\left(
 a^{\tilde{t}}\right).
$$
As $\vert \tilde{\mathbf{W}}_{\delta}^{\infty} -\mathbf{V}_{\delta}^{\infty}\vert \leq \vert \tilde{\mathbf{W}}^{\infty} -\mathbf{V}^{\infty}\vert +\delta_v +\delta_w$ and from (\ref{approximation}), applied to the points $\tilde{z}^l_m$'s instead of the points $z^l_m$'s, we have $\vert \tilde{\mathbf{W}}^{\infty} -\mathbf{V}^{\infty}\vert =O(a^{1-h})$, then
\begin{equation}\label{111}
\vert \mathbf{W}_{\delta}^{\infty} -\mathbf{V}_{\delta}^{\infty}\vert=O(a^{1-h})+\delta_w +\delta_v.
\end{equation}
What is left is the estimate of the term $\vert \textbf{C}^{-1}(\tilde{\textbf{V}}\tilde{\textbf{V}}^\top)^{-1}
\tilde{\textbf{V}}\vert$ (respectively $\vert \tilde{\textbf{V}} (\tilde{\textbf{V}}\tilde{\textbf{V}}^\top)^{-1}
 \textbf{C}^{-1}\vert$ ) or the term $\vert (\tilde{\textbf{V}}\tilde{\textbf{V}}^\top)^{-1} \tilde{\textbf{V}}
\vert $ (respectively $\vert \tilde{\textbf{V}} (\tilde{\textbf{V}}\tilde{\textbf{V}}^\top)^{-1}
\vert $ ) since the matrix $\textbf{C}=-4\pi a^{1-h}(1-a^h) I_2$. Arguing as in Lemma \ref{Matrix-V-inver}, the matrix $\tilde{\textbf{V}}\tilde{\textbf{V}}^\top$ is invertible. 
However, what we need here is its lower bound. Wet set $l_v$ its lower bound. Then $\vert \textbf{C}^{-1}(\tilde{\textbf{V}}\tilde{\textbf{V}}^\top)^{-1}
\tilde{\textbf{V}}\vert =O(a^{h-1})l^{-1}_v=\vert \tilde{\textbf{V}}(\tilde{\textbf{V}}\tilde{\textbf{V}}^\top)^{-1}
\textbf{C}^{-1}\vert$.

With these estimates, the error term in (\ref{true-formula-shifted-2}) becomes
$$
O(a^{2h-2}) O(a+a^{h-1} (\delta_u+\delta_v)+\eta)(O(a^{1-h})+\delta_w +\delta_v)+ (\delta_w+\delta_v) O(a^{2h-2})l^{-2}_v +\;~ O\left(
 a^{1-h-2\tilde{t}}+a^{h}\right)+O\left(
 a^{\tilde{t}}\right)
$$
or
$$
O(a^h)+O(a^{2h-2})(\delta_u+\delta_v)+O(a^{3h-2})(\delta_w+\delta_v)+O(a^{3h-3})(\delta_u+\delta_v)(\delta_w+\delta_v)+O(a^{h-1})\eta +O(a^{2h-2}) \eta (\delta_w+\delta_v)
$$
$$
(\delta_w+\delta_v) O(a^{2h-2})l^{-2}_v+O\left(
 a^{1-h-2\tilde{t}}\right)+O\left(
 a^{\tilde{t}}\right).
$$

If we reasonably assume that we have uniform noise for our different measurements, i.e. $\delta_u=O(\delta)$,$\delta_v=O(\delta)$ and $\delta_w=O(\delta)$, then this reduces to
$$
O(a^h)+\delta O(a^{2h-2})l^{-2}_v+O(a^{3h-3})\delta^2+O(a^{h-1})\eta +O(a^{2h-2}) \eta \delta+O\left(
 a^{1-h-2\tilde{t}}\right)+O\left(
 a^{\tilde{t}}\right).
$$
We summarize these findings in the following theorem
\begin{theorem}
Assume $n$ to be a measurable and bounded function in $\mathbb{R}^3$ such that the support of 
$(n-1)$ is a bounded domain $\Omega$. 
Let the non negative parameters $\beta, h$ and $t$ describing the set of the 
small inclusions, embedded in $\Omega$, be such that
\begin{equation}\label{cond-beta-s-h-t-Th}
 \beta =1,\; 0<h\; \mbox{ and }\; h+2t<1.
\end{equation}
We assume that the measured far-fields are contaminated with an additive noise of amplitude $\delta$, i.e. the actual measured data is $V_{\delta}^{\infty}, U_{\delta}^\infty$ and $W_{\delta, n}^\infty$ satisfying respectively
$$
\vert V_{\delta}^{\infty}(-\theta_i, \theta_j)- V^{\infty}(-\theta_i, \theta_j)\vert \leq \delta, \; \vert U_{\delta}^{\infty}(z^l_m; -\theta_i, \theta_j)-U^{\infty}(z^l_m, -\theta_i, \theta_j) \vert \leq \delta
$$
and
\begin{equation}
 \vert W_{\delta}^{\infty}\textcolor{black}{(\tilde{z}^l_m; -\theta_i, \theta_j)} -W^{\infty}\textcolor{black}{(\tilde{z}^l_m; -\theta_i, \theta_j)}\vert \leq \delta
\end{equation}
uniformly in terms of $\theta_i, \theta_j \in \mathbb{S}^2$, $i, j=1,..., N_0$. 

\begin{enumerate}
\item Let $m=1, 2,..., M$ be fixed. The vectors $\pm(V^t(z^1_m,\theta_j), V^t(z^2_m,\theta_j))$ satisfy the formulas
\begin{eqnarray}\label{V-with-noisy-data-Theorem}
 C^{-1}_{m,l} \left [U_{\delta}^\infty(z^l_m, -\theta_i,\theta_j)-V_{\delta, n}^\infty(-\theta_i,\theta_j)\right]=-V^t(z^l_m, \theta_i)V^t(z^l_m, \theta_j)+\;
O\left(a\right)+O(a^{h-1}) \delta.\quad
\end{eqnarray}

\item Let $\pm \textbf{V}_{\delta}$  be the constructed fields from step $1$. We set $\eta$ to be the maximum shift of the centers of the perturbations from step $1$ to step $2$, that we denote by $\tilde{z}^l_m, m=1,..., M$ and $l=1,2$. Let also the minimum distance between these shifted perturbations, i.e. $\tilde{d}$, be of the order $a^{\tilde{t}}$, $\tilde{t}>0$.  Using the noisy data $\tilde{\mathbf{W}}_{\delta}^{\infty}$ and $\mathbf{V}_{\delta}^{\infty}$, we have the reconstruction formulas, $l=1,2$,

\begin{equation}\label{true-formula-shifted-2-Thsas}
i\kappa n(\tilde{z}^l_m)=\\
\left[\textbf{C}^{-1}(\textbf{V}_{\delta}\textbf{V}_{\delta}^\top)^{-1}\textbf{V}_{\delta}
\left[\tilde{\mathbf{W}}_{\delta}^{\infty} -\mathbf{V}_{\delta}^{\infty}\right]\textbf{V}_{\delta}^\top(
\textbf{V}_{\delta}\textbf{V}_{\delta}^\top)^{-1}\textbf{C}^{-1}
+ \textbf{C}^{-1}\right]_{_{1,2}} -\frac{1}{\;4\pi\vert \tilde{z}^2_m-\tilde{z}^1_m\vert} 
\end{equation}
$$
+O(a^h)+\delta O(a^{2h-2})l^{-2}_v+O(a^{3h-3})\delta^2+O(a^{h-1})\eta +O(a^{2h-2}) \eta \delta+O\left(
 a^{1-h-2\tilde{t}}\right)+O\left(
 a^{\tilde{t}}\right).
$$
\end{enumerate}
\end{theorem}

The formula (\ref{true-formula-shifted-2-Thsas}) will make sense if the error term is of the order $o(1)$ as $a\rightarrow 0$.
It is clear that the noise level should decrease in terms of $a$. We set $\delta =O(a^{q_1})$ and $\eta=O(a^{q_2})$, then this error has the form 
\begin{equation}
O(a^h)+O(a^{q_1+2h-2})l^{-2}_v+O(a^{3h-3+2q_1})+O(a^{h-1+q_2})+
O(a^{2h-2+q_1+q_2})+O(a^{1-h-2\tilde{t}})+O\left(
 a^{\tilde{t}}\right).
\end{equation}
Hence the reconstruction formulas apply if we have the following relations:
\begin{equation}\label{conditions-step2}
0<h<1,\; 0<\tilde{t}<\frac{1-h}{2},\; q_1>2-2h,\; q_2>1-h
\end{equation}
even if the constant $l^{-1}_v$ might be quite large.  
The formulas (\ref{V-with-noisy-data-Theorem}), to compute the total fields, apply if 
\begin{equation}\label{conditions-step1}
0<h<1,\; q_1>1-h.
\end{equation}
We observe that the conditions on the noise level to reconstruct the index of refraction (i.e. using both single and multiple perturbations) are more restrictive than those needed to construct only the total fields (i.e. using only single perturbations as it is suggested in the previous literature), which is natural of course.  
\section{{Conclusion and comments}}
We have shown that by deforming a medium with multiple and close inclusions, we can extract the internal values of the coefficients modeling the medium from the far-field data. We have tested this idea for the acoustic medium modeled by an index of refraction where the small  inclusions are impedance type small scatterers.

\begin{enumerate}

\item Let us consider the more general model $\nabla \cdot c^2 \nabla +\kappa^2 n^2$, where $c$ is the velocity which we assume to be smooth and eventually different from a constant and the small perturbation are 
also of impedance type \footnote{As already mentioned in the point $4$ of Remark \ref{Remark}, this is possible only if the localized perturbations create high contrasts so that, by the 
Leontovich approximation, we replace the contrast by impedance type boundary conditions.}. In this case, arguing as in \cite{C-S:2016}, we can derive the same asymptotic expansion as in (\ref{x oustdie1 D_m farmain-2})-(\ref{fracqcfracmain-2})-(\ref{Cm-s-1}). Hence proceeding as in the three steps we described above, we obtain:
\begin{equation}\label{velocity-c}
\textbf{C}^{-1} (\textbf{V}\textbf{V}^\top)^{-1}\textbf{V}\left [\mathbf{W}^{\infty}-\mathbf{V}^{\infty}\right]\textbf{V}^\top(\textbf{V}\textbf{V}^\top)^{-1}\textbf{C}^{-1} 
+ \textbf{C}^{-1}=\tilde{\textbf{G}}+\;~ O\left(
 a^{1-h-2t}+a^{h}\right) 
\end{equation}  
and instead of (\ref{asymp-G}), we obtain
\begin{equation}\label{velocity-c-2}
G (x, z_m)= \left [ \frac{e^{i\kappa \frac{n(z_m}{c(z_m)})\vert x-z_m\vert}}{4\pi c^2(z_m) \vert x-z_m\vert}\right]+O(\ln(\vert x-z_m\vert)),\; \mbox{ as } \vert x-z_m\vert \rightarrow 0.
\end{equation}
From (\ref{velocity-c-2}), we derive
\begin{equation}\label{velocity-c-3}
G (x, z_m)= \left [ \frac{1}{4\pi c^2(z_m) \vert x-z_m\vert}\right]+
O(\ln(\vert x-z_m\vert)),\; \mbox{ as } \vert x-z_m\vert \rightarrow 0.
\end{equation}
Combining (\ref{velocity-c-3}) and (\ref{velocity-c}), with $x=z_m^1$ and $z_m:=z^2_m$, we derive the formula
\begin{equation}
c^{-2}(z_m)=4\pi \vert z^1_m-z^2_m\vert \left [
\textbf{C}^{-1} (\textbf{V}\textbf{V}^\top)^{-1} \textbf{V}\left[\mathbf{W}^{\infty}-\mathbf{V}^{\infty}\right]\textbf{V}^\top(\textbf{V}\textbf{V}^\top)^{-1}\textbf{C}^{-1} 
+ \textbf{C}^{-1}\right ]_{1, 2}
\end{equation}
$$
+O(a^t \ln\; a)\; ~+ O(a^{1-h-2t}+a^{h})
$$
for $i=1,...,M$, recalling that $\vert z^1_m-z^2_m\vert = a^t$, for $m=1,...,M$.

We see that we can derive a direct formula linking the measured data $\mathbf{W}^{\infty}$ and $\mathbf{V}^{\infty}$ to the value of the velocity $c$ on sampling points 
$z_m, m=1, 2,..., M$ in $\Omega$. The corresponding formulas using the noisy data can also be derived.


\item The more realistic model would be $\nabla \cdot c^2 \nabla +\kappa^2 n^2$ where the small perturbations 
occur in the coefficients $c$ and $n$ and not as small impenetrable obstacles 
as we dealt with in the present paper. One needs first to derive the corresponding 
asymptotic expansion taking into account the parameters modeling these small inclusions 
and then use them to design a reconstruction method to extract $c$ or/and $n$ from similar
measured data as we did in the present work. In the model we dealt with in this paper, we have chosen the perturbations to be spherical and the surface impedances to be of the form $\lambda_m:=(1-a^h)a^{-1}$. 
With this choice, the error in the approximation (\ref{x oustdie1 D_m farmain-3}) allows us to recover the fields created at the second order scattering (and not only the Born approximation). This is the key point in order to reconstruct the index of refraction $n$ using multiple perturbations. To extend this argument to the case when the small perturbations occur on the coefficients $c$ and $n$, one needs to derive the Foldy-Lax approximation for such a model allowing to recover fields created at least at the second order (and, again, not only the Born approximation).  This issue is related to the full justification of the Foldy-Lax approximation (i.e. at higher orders). 
 The justification of the accuracy and the optimal error estimate in the approximation of the scattered field by the Foldy-Lax field is in its generality a largely open issue but there is an increase of interest to understand it, see for instance 
 \cite{Bendali-Coquet-Tordeux:NAA2012, C-H:2013, DPC-SM13, C-S:2015, 
 M-M:MathNach2010, M-M-N:MMS:2011, M-M-N:Springerbook:2013}.

Finally, we do believe that this approach can be applied to 
other multiwave imaging modalities as well.
\end{enumerate}

\begin{center} {\bf{Acknowledgement}} \end{center}
This work was funded by the Deanship of Scientific Research (DRS), King Abdulaziz University, under the grant no. 20-130-36-HiCi. The authors, therefore, acknowledge with thanks DRS technical and financial support.

\bibliographystyle{abbrv}

\begin{thebibliography}{10}

\bibitem{BA-DPC-KMSM:JMAA:2015}
B.~Ahmad, D.~P. Challa, M.~Kirane, and M.~Sini.
\newblock The equivalent refraction index for the acoustic scattering by many
  small obstacles: with error estimates.
\newblock {\em J. Math. Anal. Appl.}, 424(1):563--583, 2015.

\bibitem{Alberti:2013}
G.~S. Alberti.
\newblock On multiple frequency power density measurements.
\newblock {\em Inverse Problems}, 29(11):115007, 25, 2013.

\bibitem{Alessandrini:2014}
G.~Alessandrini.
\newblock Global stability for a coupled physics inverse problem.
\newblock {\em Inverse Problems}, 30(7):075008, 10, 2014.

\bibitem{A:2008}
H.~Ammari.
\newblock {\em An introduction to mathematics of emerging biomedical imaging},
  volume~62 of {\em Math\'ematiques \& Applications (Berlin) [Mathematics \&
  Applications]}.
\newblock Springer, Berlin, 2008.

\bibitem{A-B-C-T-F:2008}
H.~Ammari, E.~Bonnetier, Y.~Capdeboscq, M.~Tanter, and M.~Fink.
\newblock Electrical impedance tomography by elastic deformation.
\newblock {\em SIAM J. Appl. Math.}, 68(6):1557--1573, 2008.

\bibitem{A-C-D-R-T}
H.~Ammari, Y.~Capdeboscq, F.~de~Gournay, A.~Rozanova-Pierrat, and F.~Triki.
\newblock Microwave imaging by elastic deformation.
\newblock {\em SIAM J. Appl. Math.}, 71(6):2112--2130, 2011.

\bibitem{A-G-W:2012}
H.~Ammari, J.~Garnier, and W.~Jing.
\newblock Resolution and stability analysis in acousto-electric imaging.
\newblock {\em Inverse Problems}, 28(8):084005, 20, 2012.

\bibitem{B:2014}
G.~Bal.
\newblock Hybrid inverse problems and redundant systems of partial differential
  equations.
\newblock In {\em Inverse problems and applications}, volume 615 of {\em
  Contemp. Math.}, pages 15--47. Amer. Math. Soc., Providence, RI, 2014.

\bibitem{B-B-M-T}
G.~Bal, E.~Bonnetier, F.~Monard, and F.~Triki.
\newblock Inverse diffusion from knowledge of power densities.
\newblock {\em Inverse Probl. Imaging}, 7(2):353--375, 2013.

\bibitem{Bendali-Coquet-Tordeux:NAA2012}
A.~Bendali, P.-H. Cocquet, and S.~Tordeux.
\newblock Scattering of a scalar time-harmonic wave by n small spheres by the
  method of matched asymptotic expansions.
\newblock {\em Numerical Analysis and Applications}, 5(2):116--123, 2012.

\bibitem{C-H:2013}
M.~Cassier and C.~Hazard.
\newblock Multiple scattering of acoustic waves by small sound-soft obstacles
  in two dimensions: mathematical justification of the {F}oldy-{L}ax model.
\newblock {\em Wave Motion}, 50(1):18--28, 2013.

\bibitem{C-S:2016}
D.~P. Challa and M.~Sini.
\newblock Multiscale analysis of the acoustic scattering by many scatterers of
  impedance type.
\newblock {\em Preprint,\url{arxiv:1412.0785}}.

\bibitem{DPC-SM13}
D.~P. Challa and M.~Sini.
\newblock On the justification of the {F}oldy-{L}ax approximation for the
  acoustic scattering by small rigid bodies of arbitrary shapes.
\newblock {\em Multiscale Model. Simul.}, 12(1):55--108, 2014.

\bibitem{C-S:2015}
D.~P. Challa and M.~Sini.
\newblock The Foldy-Lax approximation of the scattered waves by many small
  bodies for the Lam\'e system.
\newblock {\em Mathematische Nachrichten}, 288(16):1834--1872, 2015.

\bibitem{C-K:1998}
D.~Colton and R.~Kress.
\newblock {\em Inverse acoustic and electromagnetic scattering theory},
  volume~93 of {\em Applied Mathematical Sciences}.
\newblock Springer-Verlag, Berlin, second edition, 1998.

\bibitem{C-K:1983}
D.~L. Colton and R.~Kress.
\newblock {\em Integral equation methods in scattering theory}.
\newblock Pure and Applied Mathematics (New York). John Wiley \& Sons Inc., New
  York, 1983.
\newblock A Wiley-Interscience Publication.

\bibitem{F-T:2010}
M.~Fink and M.~Tanter.
\newblock Multiwave imaging and super resolution.
\newblock {\em Physics Today}, 63(2):28--33, 2010.

\bibitem{F-T:2011}
M.~Fink and M.~Tanter.
\newblock A multiwave imaging approach for elastography.
\newblock {\em Current Medical Imaging Reviews}, 7(4):340--349, 2011.

\bibitem{H-M-N:2014}
N.~Honda, J.~McLaughlin, and G.~Nakamura.
\newblock Conditional stability for a single interior measurement.
\newblock {\em Inverse Problems}, 30(5):055001, 19, 2014.

\bibitem{K-G}
A.~Kirsch and N.~Grinberg.
\newblock {\em The factorization method for inverse problems}, volume~36 of
  {\em Oxford Lecture Series in Mathematics and its Applications}.
\newblock Oxford University Press, Oxford, 2008.

\bibitem{L-F:2013}
G.~Lerosey and M.~Fink.
\newblock Acousto-optic imaging: Merging the best of two worlds.
\newblock {\em Nature Photonics}, 7(4):265--267, 2013.

\bibitem{M-M:MathNach2010}
V.~Maz\'ya and A.~Movchan.
\newblock Asymptotic treatment of perforated domains without homogenization.
\newblock {\em Math. Nachr.}, 283(1):104--125, 2010.

\bibitem{M-M-N:MMS:2011}
V.~Maz\'ya, A.~Movchan, and M.~Nieves.
\newblock Mesoscale asymptotic approximations to solutions of mixed boundary
  value problems in perforated domains.
\newblock {\em Multiscale Model. Simul.}, 9(1):424--448, 2011.

\bibitem{M-M-N:Springerbook:2013}
V.~Maz\'ya, A.~Movchan, and M.~Nieves.
\newblock {\em Green's Kernels and Meso-Scale Approximations in Perforated
  Domains}, volume 2077 of {\em Lecture Notes in Mathematics}.
\newblock Springer-Verlag, Berlin, 2013.

\bibitem{Mclean:2000}
W.~McLean.
\newblock {\em Strongly elliptic systems and boundary integral equations}.
\newblock Cambridge University Press, Cambridge, 2000.

\bibitem{Nachman:1988}
A.~I. Nachman.
\newblock Reconstructions from boundary measurements.
\newblock {\em Ann. of Math. (2)}, 128(3):531--576, 1988.

\bibitem{Novikov:1988}
R.~G. Novikov.
\newblock A multidimensional inverse spectral problem for the equation
  {$-\Delta\psi +(v(x)-Eu(x))\psi=0$}.
\newblock {\em Funktsional. Anal. i Prilozhen.}, 22(4):11--22, 96, 1988.

\bibitem{Ramm:1988}
A.~G. Ramm.
\newblock Recovery of the potential from fixed-energy scattering data.
\newblock {\em Inverse Problems}, 4(3):877--886, 1988.

\bibitem{MRAMAG-book1}
A.~G. Ramm.
\newblock {\em Wave scattering by small bodies of arbitrary shapes}.
\newblock World Scientific Publishing Co. Pte. Ltd., Hackensack, NJ, 2005.

\bibitem{Triki:2010}
F.~Triki.
\newblock Uniqueness and stability for the inverse medium problem with internal
  data.
\newblock {\em Inverse Problems}, 26(9):095014, 11, 2010.

\bibitem{Val}
N.~Valdivia.
\newblock Uniqueness in inverse obstacle scattering with conductive boundary
  conditions.
\newblock {\em Appl. Anal.}, 83(8):825--851, 2004.

\end{thebibliography}
 
 \def\cprime{$'$}

\end{document}